\documentclass[11pt]{amsart}
\usepackage{geometry,a4wide,amssymb,amsfonts,amsrefs,amsmath}                
\usepackage{bbm,xcolor}
\usepackage{dsfont}
\usepackage{yfonts}
\usepackage[T1]{fontenc}
\usepackage{mathtools}
\usepackage{mathrsfs}
\usepackage{hyperref}  
\usepackage{footmisc} 

\newcommand{\Z}{\mathbb{Z}}
\newcommand{\z}{\mathcal{Z}}

\newcommand{\dom}{\mathtt{dom}}
\newcommand{\A}{\mathcal{A}}
\newcommand{\ca}{\mathscr{A}}
\newcommand{\N}{\mathbb{N}}
\newcommand{\bL}{\mathbb{L}}
\newcommand{\R}{\mathbb{R}}
\newcommand{\F}{\mathcal{F}}
\newcommand{\s}{\mathfrak{S}}

\newcommand{\M}{\mathcal{M}}
\newcommand{\G}{\mathscr{G}}

\newcommand{\La}{\mathcal{L}}

\theoremstyle{plain}

\newtheorem{thm}{Theorem}[section]

\newtheorem{prop}[thm]{Proposition}

\newtheorem{defin}[thm]{Definition}

\newtheorem{remark}[thm]{Remark}

\newtheorem{manualtheoreminner}{Theorem}
\newenvironment{manualtheorem}[1]{%
  \renewcommand\themanualtheoreminner{#1}%
  \manualtheoreminner
}{\endmanualtheoreminner}

\newtheorem{manualpropinner}{Proposition}
\newenvironment{manualprop}[1]{%
  \renewcommand\themanualpropinner{#1}%
  \manualpropinner
}{\endmanualpropinner}

\newtheorem{manualleminner}{Lemma}
\newenvironment{manuallem}[1]{%
  \renewcommand\themanualleminner{#1}%
  \manualleminner
}{\endmanualleminner}

\title[Dyson models: Transfer operators and Concentration inequalities]
{Concentration inequalities  and Transfer operators for supercritical Dyson models}

\author{Mirmukhsin Makhmudov}

\address{Mathematical Institute,  Leiden University, Einsteinweg 55, 2333 CC Leiden, The Netherlands}
\email{m.makhmudov@math.leidenuniv.nl}

\begin{document}

\begin{abstract}
    The present paper extends the results on the ferromagnetic Dyson models from \cites{EFMV2024, JOP2025} 
    to the near-critical and strongly interacting regimes.
    As part of our main result, we further establish Gaussian concentration bounds for the unique infinite-volume Gibbs measure throughout the entire supercritical regime. 
\end{abstract}

\maketitle

\section{Introduction}

The one-dimensional long-range Ising model, a.k.a. the Dyson model, is one of the archetypical models in Statistical Mechanics. 
This model has served as the testing ground for many conjectures and ideas in statistical mechanics, and it has also been the first example of a one-dimensional model which undergoes phase transitions at low temperatures. 
A Dyson interaction $\Phi=(\Phi_\Lambda)_{\Lambda\Subset\Z}$ is defined on the configuration space $X:=\{\pm 1\}^\Z$ of $\omega\in X$ by:
\begin{equation}\label{eq: Dyson interaction}
    \Phi_{\Lambda}(\omega):=\begin{cases} 
      -J(|i-j|)\omega_i\omega_j, & \text{ if }\Lambda=\{i,j\}\subset\Z, \ i\neq j;\\
      0, & \text{otherwise},
   \end{cases}
\end{equation}
where $J(k)\geq 0$ for every $k\in\N$ and to have well-defined infinite-volume Gibbs measures, we assume that $\sum_{k=1}^\infty J(k)<\infty$.
The most well-known example is the \textit{standard Dyson interaction} for which $J(k)=k^{-\alpha}$ for every $k\in\N$ and for some $\alpha>1$.
In this paper, instead of studying one fixed interaction $\Phi$, we study the family $\beta\Phi=(\beta\Phi_\Lambda)_{\Lambda\Subset\Z}$ of interactions parameterised by $\beta\geq 0$, which is interpreted as the inverse temperature in statistical mechanics.
A \textit{Gibbs measure} for the interaction $\beta\Phi$ is a probability measure $\mu$ on $X$ that satisfies the Dobrushin–Lanford–Ruelle (DLR) equations.
Equivalently, for finite volumes $\Lambda\Subset\Z$, the conditional distribution of $\mu$ given the configuration outside $\Lambda$ is prescribed by the Boltzmann-Gibbs ansatz:
\begin{equation}
    \mu(\omega_\Lambda \mid \omega_{\Lambda^c}) 
    =
    \gamma^{\beta\Phi}_\Lambda(\omega_\Lambda \mid \omega_{\Lambda^c}),\;\;\; \mu-\text{a.e. } \;\omega\in X,
\end{equation}
where 
\begin{equation}
    \gamma^{\beta\Phi}_\Lambda(\omega_\Lambda \mid \omega_{\Lambda^c})
    =
    \frac{\exp\left(-H_\Lambda^{\beta\Phi}(\omega)\right)}{\sum\limits_{\bar\omega_\Lambda\in \{\pm1\}^\Lambda} \exp\left(-H_\Lambda^{\beta\Phi}(\bar\omega_\Lambda \omega_{\Lambda^c})\right)},\;\;\; \;
    H_\Lambda^{\beta\Phi}=\sum_{\substack{V: V\Subset\Z\\ V\cap\Lambda\neq \emptyset}} \beta\Phi_V,
\end{equation}
and $\bar\omega_\Lambda\omega_{\Lambda^c}\in X$ denotes the configuration constructed by concatenating the strings $\bar\omega_\Lambda$ and $\omega_{\Lambda^c}$.
It is a well-known fact \cites{Dyson1969, FS1982} for the standard Dyson model that for every $\alpha\in (1,2]$ there exists a critical value $\beta_c(\alpha)\in(0,\infty)$ such that for every $\beta<\beta_c(\alpha)$, there exists a unique Gibbs measure for $\beta\Phi$ and for every $\beta>\beta_c(\alpha)$, there are multiple Gibbs measures for $\beta\Phi$.

We also consider the interaction $\Phi$ on the half-line $\Z_+$.
Studying the interaction $\Phi$ on ${\Z_+}$ amounts to studying the thermodynamic properties of the potential:
\begin{equation}
    \phi(x):=\sum_{n=1}^\infty J(n)x_0 x_n, \;\;\; x\in X_+:=\{\pm1\}^{\Z_+}.
\end{equation}
Note that $\phi$ and $\Phi$ are related through:
\begin{equation}
    \phi=-\sum_{0\in V\Subset \Z_+} \Phi_V,
\end{equation}
which often appears in the variational characterisation of the Gibbs measures.
In fact, by the Gibbs variational principle \cite{Georgii-book}*{Theorem 15.39}, the translation-invariant Gibbs measures for $\Phi$ -- that is, those preserved by the shift map $S$ -- are precisely the \textit{equilibrium states} for the potential $\phi$. These are the measures $\mu\in\M_{1,S}(X)$ satisfying 
\begin{equation}
    h(\mu)+\int_X \phi d\mu=\sup\Big\{h(\tau)+\int_X \phi d\tau: \tau\in\M_{1,S}(X)\Big\},
\end{equation}
where $h(\tau)$ is the measure-theoretic entropy of the measure $\tau$, the shift map $S:X\to X$ is defined by $S\omega=(\omega_{i+1})_{i}$ and $\M_{1,S}(X)$ denotes the set of the translation-invariant probability measures on $X$.

One of the main objectives in this paper is to study the Ruelle-Perron-Frobenius \textit{transfer operator} $\La_{\beta\phi}$ for the potential $\beta\phi$, which acts on the space of functions on $X_+$ by
\begin{equation}\label{eq: def of transfer operator}
    \La_{\beta\phi} f(x):=\sum_{a\in \{\pm1\}} e^{\beta\phi(ax)}f(ax), \;\; x\in X_+,
\end{equation}
here $f\in \R^{X_+}$ and $ax\in X_+$ denotes concatenated configuration.
First, we note that for the spectral radius $\rho(\La_{\beta\phi})$ of the operator $\La_{\beta\phi}$, one has $\rho(\La_{\beta\phi})=e^{P(\beta\phi)}$, where $P(\beta\phi)$ is the topological pressure of $\beta\phi$, which also coincides with the infinite-volume pressure of the interaction $\beta\Phi$.
Second, the adjoint operator $\La_{\beta\phi}^*$, which acts on the space of signed measures $\M(X_+)$, always has an eigenprobability $\nu\in\M_1(X_+)$ corresponding to the spectral radius $\lambda=\rho(\La_{\beta\phi})$, i.e., $\La_{\beta\phi}^*\nu=\lambda \nu$.
Throughout the paper, we also refer to the eigenprobabilities of $\La_{\beta\phi}^*$ as the one-sided Gibbs measures for $\beta\phi$ (or $\beta\Phi$) since they are exactly the probability measures satisfying the DLR equations corresponding to $\beta\phi$ (or $\beta\Phi|_{\Z_+}$) \cite{CLS2020}.
In the half-line $\Z_+$, a similar phase diagram to that of the interactions $\beta\Phi$ on $\Z$ is also established.
Indeed, in \cite{JOP2019}, it was proven for the standard Dyson model that there exists a critical temperature $\beta_c^+(\alpha)\in (0, \infty)$ such that for every $\beta\in [0,\beta^+_c(\alpha))$ 
there exists a unique eigenprobability of $\La_{\beta\phi}^*$ and for $\beta>\beta^+_c(\alpha)$ there are multiple eigenprobabilities. 
In \cite{JOP2019}, Johansson, \"Oberg and Pollicott also showed that for every $\alpha\in (1, 2]$, $\beta_c^+(\alpha)\leq 8 \beta_c(\alpha)$ and conjectured that $\beta^+_c(\alpha)=\beta_c(\alpha)$.
We note that part of this conjecture -- namely, that $\beta_c(\alpha) \leq \beta_c^+(\alpha)$ for every $\alpha \in (1, 2]$ -- follows directly from the Griffiths inequalities \cites{FV-book, Griffiths1966, Griffiths1968} (see Proposition \ref{prop: critical beta's for dyson and intermediate interactions}). The eigenprobabilities correspond, in a dual sense, to the principal eigenfunctions of the transfer operator $\La_{\beta\phi}$. While the existence of an eigenprobability of the adjoint operator $\La_{\beta\phi}^*$ follows directly from the DLR equations, establishing the existence of a principal eigenfunction for $\La_{\beta\phi}$ is considerably more subtle and technically involved.
Important advancements in this direction have been made in~\cites{EFMV2024, JOP2025, M2025}, whose relevant results — specialised here to the standard Dyson model — are summarised in the following theorem.
\begin{thm}\cites{EFMV2024, JOP2025}\label{EFMV2024 and JOP2025 results about Dyson model} Let $\Phi$ be the standard Dyson interaction (\ref{eq: Dyson interaction}).
Then
\begin{itemize}
    \item[(i)] for all $\alpha>1$ and for $\beta\geq 0$ \underline{sufficiently small}, the unique Gibbs state $\mu$ (in fact, the restriction to $X_+$) for $\beta\Phi$ is equivalent to the unique half-line Gibbs state $\nu$ for $\beta\phi$, i.e., $\mu\ll\nu$ and $\nu\ll\mu$.
    In particular, the Perron-Frobenius transfer operator $\La_{\beta\phi}$ has a non-negative integrable eigenfunction in $L^1(X_+,\nu)$;  
    
    \item[(ii)] if $\alpha>\frac{3}{2}$, then for all $0\leq\beta<\beta_c(\alpha)$, the Gibbs state $\mu$ for $\beta\Phi$ 
    is not only equivalent to the half-line Gibbs state $\nu$,  but there even exists 
    a continuous version of the Radon-Nikodym density $\frac{d\mu}{d\nu}$.
    Thus, the Perron-Frobenius transfer operator $\La_{\beta\phi}$ has a positive continuous eigenfunction. 
\end{itemize}
\end{thm}
\noindent
The main objective of this paper, in the context of transfer operators, is to extend the above result to the entire supercritical region, that is, to all pairs $(\alpha, \beta)$ with $\alpha \in (1,2]$ and $\beta \in [0, \beta_c(\alpha))$.
We, in fact, state the main result of this paper for interactions given by~\eqref{eq: Dyson interaction}, which are slightly more general than the standard Dyson interaction. 
A broader generalisation was considered in~\cite{EFMV2024}, where the authors studied a class of long-range interactions -- potentially including multi-body and non-ferromagnetic cases.
%-- which is more general than the ferromagnetic Dyson interactions considered in this paper.
However, the analysis in~\cite{EFMV2024} was restricted to sufficiently small values of $\beta\geq 0$, ensuring that the interaction $\Phi$ satisfies the Dobrushin uniqueness condition~\cite{Georgii-book}*{Chapter 8}.
This condition corresponds to the high-temperature regime and is known to be strictly contained within the uniqueness region.
\textit{Near-critical regimes}, i.e., when $0 \ll \beta < \beta_c(\alpha)$, are typically more delicate to analyse in statistical mechanics and often require different techniques. 
This phenomenon is also reflected in our setting. 
We also note that \cite{JOP2025} approached the problem using the random cluster representation -- a method fundamentally different from the techniques used in \cite{EFMV2024} -- and successfully treated both the high-temperatures $0 \leq \beta \ll \beta_c(\alpha)$ and the near-critical regimes $0 \ll \beta < \beta_c(\alpha)$, under the assumption that $\alpha > \frac{3}{2}$.
However, the methods in \cite{JOP2025} fundamentally depend on the square summability of the variations of the potential $\phi$, a condition that holds only for $\alpha > \frac{3}{2}$. 
Consequently, their approach does not extend to the strongly interacting setting $\alpha \leq \frac{3}{2}$.
In another related work \cite{M2025}, the principal eigenfunctions were studied in the presence of an external field in (\ref{eq: Dyson interaction}). While that work addresses a different setting -- since we focus on the case without external fields -- it is noteworthy that the presence of an external field, as shown in \cite{M2025}, tends to worsen the regularity of the eigenfunction.

Another objective of this paper is to study the concentration properties of the Gibbs measures for the near-critical Dyson models.
This has been done for the two-dimensional nearest-neighbour Ising model \cites{Moles-PhD-thesis,CMRU2020} using methods very specific for the short-range models and the two-dimensional lattice $\Z^2$.
The problem has been remaining open for the long-range models, and recently, Bauerschmidt and Dagallier \cite{BD2024} succeeded in obtaining the log-Sobolev inequalities (see Section \ref{section about LSI} for the detailed discussion) for the finite-volume Gibbs measures of ferromagnetic Ising models.
In this paper, we employ Bauerschmidt and Dagallier's result to establish the \textit{\textbf{Gaussian Concentration Bounds}} (GCB) for the unique \textit{infinite-volume Gibbs measure} of the interaction $\beta\Phi$ for all the values of $0\leq \beta<\beta_c(\alpha)$.
We note that the Gaussian Concentration Bound for a measure $\mu\in\M_1(X)$ (or $\mu\in\M_1(X_+)$) refers to a concentration inequality of the form
\begin{equation}\label{eq: def of GCB with tails}
        \mu\Big(\Big\{\omega\in X: \; F(\omega)-\int_{X} F d\mu\;\geq\; t\Big\}\Big)\;\leq\; e^{-\frac{2t^2}{D||\underline{\delta}(F)||_2^2}},\;\;\; t>0,
    \end{equation}
where $F$ is a real-valued continuous function on $X$, $D>0$ is a constant independent of $F$ and $t$, and $\lVert\underline{\delta}(F)\rVert^2_{2}$ denotes the total oscillation of $F$, i.e., $\lVert\underline{\delta}(F)\rVert^2_{2}:=\sum_{k\in\Z} (\delta_kF)^2$ and $\delta_kF:=\sup\bigl\{F(\xi)-F(\eta):\xi_j=\eta_j,\; j\in\Z\setminus\{k\}\bigr\}$.
It is a well-known fact (\cite{V2018}*{Proposition 2.5.2}) that up to the constant $D$, (\ref{eq: def of GCB with tails}) is equivalent to $F:X\to\R$ to be \textit{sub-Gaussian}, i.e., 
\begin{equation}\label{GCB ineq}
     \int_{X}e^{F-\int_{X}Fd\mu}d\mu\leq e^{D||\underline{\delta}(F)||^2_{2}}.
\end{equation}
or $\mu$ to have \textit{moment concentration bounds}, i.e., for all $m\in\N$, 
\begin{equation}\label{eq: MCB}
       \int_{X} \Big|F-\int_{X}Fd\mu\Big|^m d\mu
       \leq \Big(\frac{D||\underline \delta (F)||_2^2}{2} \Big)^{\frac{m}{2}} m \Gamma\Big(\frac{m}{2}\Big),
\end{equation}
where $\Gamma$ is Euler's gamma function.     

We now summarise the main results of this paper in the following theorem.
\begin{manualtheorem}{A}\label{main result}
Let $\Phi$ be the Dyson interaction given in (\ref{eq: Dyson interaction}), with couplings $J(k) \geq 0, \; k\in\N$ such that $\sum_{k=1}^\infty J(k) < \infty$.
\begin{itemize}
    \item[(i)] If $\sum_{i=1}^\infty\sum_{k=i}^\infty J(k)^2<\infty$, then for every $\beta < \beta_c(\Phi)$, the Gaussian Concentration Bound (GCB) holds for the unique Gibbs measure $\mu$ of the interaction $\beta\Phi$, and for the unique eigenprobability $\nu$ of the transfer operator $\La_{\beta\phi}$, here $\beta_c(\Phi)$ is the critical inverse temperature of the phase transitions for $\Phi$, i.e.,
    \[
    \beta_c(\Phi):=\sup\{\beta\geq 0: \beta\Phi \text{ has a unique Gibbs measure}\}.
    \]
        
    \item[(ii)]
    If 
    $J(k)=O(k^{-1})$, then for every $\beta < \beta_c(\Phi)$, the transfer operator $\La_{\beta\phi}$ admits a non-negative principal eigenfunction in $L^1(\nu)$.
    
    \item[(iii)] If $\sum_{i=1}^\infty \Big(\sum_{k=i}^\infty J(k)\Big)^2<\infty$, then for every $\beta < \beta_c(\Phi)$, the transfer operator $\La_{\beta\phi}$ admits a positive continuous principal eigenfunction.
    \end{itemize}
\end{manualtheorem}
\begin{remark}
    \begin{itemize}
        \item[(1)] The couplings in the standard Dyson interaction satisfy conditions (i) and (ii) of the theorem stated above for all values of $\alpha > 1$.
        However, condition in (iii) is satisfied if and only if $\alpha > \frac{3}{2}$. 
        In \cite{EFMV2024}, the authors conjectured that the principal eigenfunction exhibits essential discontinuities when $\alpha < \frac{3}{2}$. 
        We maintain this conjecture,
        but do not address it in the present work.    
        \item[(2)] We expect that the second part of Theorem \ref{main result} remains valid under the slightly weaker condition in part (i), namely that $\sum_{i=1}^\infty \sum_{k=i}^\infty J(k)^2 < \infty$. This is supported by a widely held expectation in statistical mechanics: that for long-range models, the two-point correlation functions decay at the same rate as the interaction energy throughout the entire supercritical region. 
        This decay behaviour can be rigorously established when $\beta$ is sufficiently small so that $\beta\Phi$ satisfies the Dobrushin uniqueness condition. In the near-critical regime, partial progress has been made; see \cites{SZ1999, A2021}.

        \item[(3)] 
        The condition $\sum_{i=1}^\infty \Big(\sum_{k=i}^\infty J(k)\Big)^2 < \infty$ also appears in \cite{JOP2025}, where, under this assumption, the authors establish the existence of a continuous eigenfunction using the random cluster representation of the Ising model—a method distinct from the one employed in the present work. 

       \item[(4)] It is natural to ask what occurs when $\beta \geq \beta_c(\alpha)$. In the phase coexistence region $\beta > \beta_c(\alpha)$, it is known from \cite{CMRU2020} that none of the extremal Gibbs states for $\beta\Phi$ satisfy the GCB.
       At criticality ($\beta = \beta_c(\alpha)$), ferromagnetic spin models often admit a unique Gibbs measure; however, this uniqueness does not imply that the GCB holds. For example, the two-dimensional nearest-neighbour Ising model \cite{CMRU2020} has a unique critical Gibbs measure, which does not satisfy the GCB.
       We expect an analogous phenomenon to occur for the Dyson model (\ref{eq: Dyson interaction}) at criticality.    \end{itemize}

\end{remark}

We prove Theorem \ref{main result} by employing the method of intermediate interactions, originally introduced in \cite{EFMV2024}.
In Section \ref{Sec: Intermediate interactions and their phase diagrams}, we construct those interactions and discuss their phase diagrams. 

In Section \ref{section about LSI}, we study the log-Sobolev inequalities for the Gibbs measures of the interaction $\Phi$ and also the intermediate interactions, by relying on a recent result by Bauerschmidt and Dagallier \cite{BD2024}. 
There, we prove Theorem \ref{thm: LSI for infnty volume Gibbs meas with uniform constant}, which establishes the log-Sobolev inequalities with a uniform log-Sobolev constant for a family of Gibbs measures including $\mu$ and $\nu$.
We then derive the GCB from these inequalities via the Herbst argument, thus proving part (i) of Theorem \ref{main result}.
As we have not found a full version of the Herbst argument suitable for our setting in the literature, we include a self-contained proof in the Appendix.

In Section \ref{Proof of the main result}, we prove the second and third parts of Theorem \ref{main result}.
We show the existence of an integrable eigenfunction by adapting the arguments from the proof of Theorem E in \cite{EFMV2024}. 
However, the proof in \cite{EFMV2024} relies on the Dobrushin uniqueness condition, which ensures strong mixing and decay of correlations but fails to hold near criticality.
To overcome this, we instead estimate the total sum of spin-spin correlations, using a result by Duminil-Copin and Tassion \cite{D-CT2016}, which serves our purpose just as well.
The proof of the continuity of the eigenfunction under slightly stronger condition $\sum_{i=1}^\infty\Big(\sum_{k=i}^\infty J(k)^2\Big)<\infty$ replicates the arguments previously used in the third part of Theorem C in \cite{EFMV2024}. 
For the sake of completeness, we include the main steps of the proof.

\section{Intermediate interactions and their phase diagrams}\label{Sec: Intermediate interactions and their phase diagrams}
\subsection{Intermediate interactions}
We start this section by recalling the construction of those intermediate interactions for the Dyson interaction $\Phi$ given by (\ref{eq: Dyson interaction}).
We will represent $\Z=-\mathbb N\cup \Z_+$.
Consider the countable collection of finite subsets of $\Z$ such that:
$$
\mathcal A=\{\Lambda\Subset \Z:\, \min(\Lambda)<0,\ \max \Lambda\ge 0\}.
$$
Index elements of $\mathcal A$ in an order $\mathcal A=\{\Lambda_1,\Lambda_2,\ldots\}$ so that for every $N\in\N$, there exists $k_N\in\N$ satisfying
\begin{equation}\label{indexation rule of mathcal{A}}
\sum_{i=1}^{k_N}\Phi_{\Lambda_i}
=
\sum_{\substack{
\min V<0\leq \max V\\
V\subset [-N,N]}} \Phi_V.       
\end{equation}
Then define
$$
{\Psi}^{(k)}_\Lambda = \begin{cases}  \Phi_{\Lambda},&\quad \Lambda\notin\{\Lambda_i: i\geq k+1\},\\
0, &\quad \Lambda\in\{\Lambda_i: i\geq k+1\},
\end{cases}\;
$$
In other words, we first remove all $\Phi_\Lambda$'s with $\Lambda\in\mathcal A$ from $\Phi$, and then add them one by one. 
Clearly, all the constructed interactions are UAC. 

\begin{remark}\label{convergence of intermediate interactions and specs }
    \begin{itemize}
        \item[1)] Every $\Psi^{(k)}$ is a local (finite) perturbation of $\Psi^{(0)}$, and $\Psi^{(k)}$ tend to $\Phi$
        as $k\to\infty$, in the sense that 
        $\Psi^{(k)}_{\Lambda}\rightrightarrows\Phi_{\Lambda}$ 
        for all $\Lambda\Subset\Z$.
        \item[2)] For specifications, it can also be concluded that for every $\beta\geq 0$ $\gamma^{\beta\Psi^{(k)}}$ converges to $\gamma^{\beta\Phi}$ as $k\to\infty$, i.e., for all $B\in\F$ and $V\Subset\Z$,
        $$
\gamma^{\beta\Psi^{(k)}}_V(B|\omega)\xrightarrow[k\to\infty] 
    {}\gamma^{\beta\Phi}_V(B|\omega)\; \text{ 
        uniformly on the b.c. } 
        \omega\in X.
        $$
        \item[3)] In addition, if $\nu^{(k)}$ is a Gibbs measure for $\beta\Psi^{(k)}$, then by Theorem 4.17 in \cite{Georgii-book}, any weak$^*$-limit point of the sequence $\{\nu^{(k)}\}_{k\geq 0}$ becomes a Gibbs measure for the potential $\beta\Phi$.

    \end{itemize}
\end{remark}

Another important observation is the following: since we have constructed $\beta\Psi^{(0)}$ from $\beta \Phi$ by removing all the interactions between the \textit{left} $-\mathbb N$ and the \textit{right} $\Z_+$ half-lines, the corresponding specification $\gamma^{\beta\Psi^{(0)}}$ becomes \textit{product type} \cite{Georgii-book}*{Example 7.18}. 
More precisely, $\gamma^{\beta\Psi^{(0)}}=\gamma^{\beta\Phi|_{-\N}}\times\gamma^{\beta\Phi|_{\Z_+}}$, where $\Phi|_{-\N}$ and $\Phi|_{\Z_+}$ are the restrictions of $\Phi$ to the half-lines $-\N$ and $\Z_+$, respectively.
Thus we have the following for the extremal Gibbs measures \cite{Georgii-book}*{Example 7.18}:

\begin{equation}\label{extreme Gibbs measures for Psi^0}
    \text{ex}\;\G(\gamma^{\beta\Psi^{(0)}})
    =\{\nu^l\times\nu^r: \nu^l\in \text{ex}\;\G(\gamma^{\beta\Phi|_{-\N}}),\;\; \nu^r\in \text{ex}\;\G(\gamma^{\beta\Phi|_{\Z_+}}) \}
\end{equation}
For each $\Psi^{(k)}$, we define the critical inverse temperature of phase transitions by 
\begin{equation*}
    \beta_c(\Psi^{(k)}):=\sup\{\beta\geq 0: \beta\Psi^{(k)} \text{ has a unique Gibbs measure}\}.
\end{equation*}
By employing the Griffiths inequalities \cites{FV-book, Griffiths1968, Griffiths1966}, one can establish the proposition below regarding the phase diagrams of the intermediate interactions $\Psi^{(k)}$ associated with the Dyson interaction $\Phi$.
\begin{prop}\label{prop: critical beta's for dyson and intermediate interactions}
For every $k\in\Z_+$,
    \begin{equation}\label{critical beta's for Phi and intermed interac}
        \beta_c(\Psi^{(k)})\geq \beta_c(\Psi^{(k+1)})\geq \beta_c(\Phi).
    \end{equation}
\end{prop}

\section{Log Sobolev inequalities for the infinite-volume Gibbs measures}\label{section about LSI}
In this section, we discuss log-Sobolev inequalities for the unique infinite-volume Gibbs measures $\mu^{\beta}$ and $\nu^{(k),\beta}$ associated with the interactions $\beta\Phi$ and $\beta\Psi^{(k)}$, respectively, for $k \in \Z_+$, in the regime $\beta < \beta_c(\Phi)$.
To emphasise the dependence on the inverse temperature, we include $\beta$ in the superscripts throughout this section.

Consider a subset (finite or infinite) $L\subseteq \Z$ and denote the corresponding configuration space by $X_L:=\{\pm 1\}^L$.
For a configuration $\omega\in X_L$, $\omega^{(i)}\in X_L$ denotes the configuration obtained from $\omega$ by flipping the spin at site $i\in\bL$ and keeping all other spins unchanged, i.e.,
    $$
        \omega^{(i)}_j := \begin{cases}
            \omega_j, & i \neq j;\\
            -\omega_i, & i = j.
        \end{cases}
    $$
\begin{defin}
    A measure $\tau$ on $X_L$ satisfies the \textit{\textbf{log-Sobolev inequality}} (LSI) if there exists a constant $D = D(\tau) > 0$ such that, for every local function $f : X_L\to \R$,
    \begin{equation}\label{eq: LSI in gen form}
        \texttt{Ent}_{\tau}(f^2)\leq 2D \int_{X_L} \sum_{i\in L} \; (f(\omega)-f(\omega^{(i)}))^2\tau(d\omega),
    \end{equation}
    where for a non-negative local function $\tilde f : X_L \to \R$, the entropy of $\tilde f$ with respect to $\tau$ is defined as
    $
    \texttt{Ent}_\tau(\tilde f) := \int_{X_L} \tilde f \log \tilde f \, d\tau 
    - \int_{X_L} \tilde f \, d\tau \cdot \log \int_{X_L} \tilde f \, d\tau.
    $
    The smallest such constant $D$ is called the \textit{log-Sobolev constant} for $\tau$.
\end{defin}

Before stating the main theorem of this section, we note that, thanks to Theorem 2.1 in \cite{D-CT2016}, the infinite-volume susceptibility of the Dyson model $\beta\Phi$, defined as
    \begin{equation}
        \chi_{\beta}(\Phi):=\sum_{i\in\Z}\mu^\beta(\sigma_0\sigma_i),
    \end{equation} 
is finite for every $0\leq \beta<\beta_c(\Phi)$, here $\mu^\beta$ is the unique infinite-volume Gibbs measure for $\beta\Phi$.

\begin{thm}\label{thm: LSI for infnty volume Gibbs meas with uniform constant}
    Assume $\beta<\beta_c(\Phi)$ and let $\nu^{(k), \beta}$ be the unique infinite-volume Gibbs measures for $\beta\Psi^{(k)}$. 
    Then $\mu^\beta$ and $\nu^{(k), \beta}$ satisfy the log-Sobolev inequality with a constant not larger than $\frac{1}{4}+\frac{\beta}{2}\exp{(2\beta \chi_{\beta}(\Phi))}$. 
\end{thm}

To prove Theorem \ref{thm: LSI for infnty volume Gibbs meas with uniform constant}, we rely on the main result in \cite{BD2024},
which is built upon a recently developed correlation inequality in \cite{DSS2023}.
In \cite{BD2024}, Bauerschmidt and Dagallier established log-Sobolev inequalities for finite-volume Gibbs measures of the ferromagnetic Ising model with free boundary conditions.
The interactions considered in \cite{BD2024} are more general than the Dyson interaction (\ref{eq: Dyson interaction}).

Below, we recall the main result of \cite{BD2024}.

Fix a finite volume $\Lambda\Subset\Z$ and let  $\ca^{(\Lambda)}$ be an arbitrary $\Lambda\times\Lambda$ real matrix satisfying the following conditions:
\begin{itemize}
    \item[(C1)] $\ca^{(\Lambda)}$ is a symmetric matrix with $\ca^{(\Lambda)}_{ij}\leq 0$ for all off-diagonal elements $i,j\in\Lambda$;
    \item[(C2)] $\ca^{(\Lambda)}$ is a positive definite matrix, i.e., $(u_\Lambda, \ca^{(\Lambda)}u_\Lambda)> 0$ for all $u_\Lambda\in\R^\Lambda$ with $u_\Lambda\neq 0$;
    \item[(C3)] for the spectral radius $\rho(\ca^{(\Lambda)})$ of $\ca^{(\Lambda)}$, one has $\rho(\ca^{(\Lambda)})\leq 1$.
\end{itemize}
One can associate a Gibbs measure with the matrix $\ca^{(\Lambda)}$ on the configuration space $X_\Lambda:=\{\pm1\}^\Lambda$ as
\begin{equation}\label{Boltzmann ansatz for matrix A}
    \tau_\Lambda^{\beta}(\{\omega_\Lambda\}) = \frac{e^{-\frac \beta 2 (\omega_\Lambda, \ca^{(\Lambda)} \omega_\Lambda)}}
    {\sum_{\omega_\Lambda \in X_\Lambda} e^{-\frac \beta 2 (\omega_\Lambda, \ca^{(\Lambda)} \omega_\Lambda)}},\;\; \omega_\Lambda \in X_\Lambda,
\end{equation}
where $\beta\geq 0$ represents the inverse temperature.
\begin{thm}\cite{BD2024}*{Theorem 1.1}\label{thm: BD result} Assume $\ca=\ca{(\Lambda)}$ is a real $\Lambda\times\Lambda$ matrix satisfying (C1)-(C3) conditions. 
Then for all $\beta\geq 0$, the finite-volume Gibbs measure $\tau^\beta_\Lambda$ satisfies the log-Sobolev inequality (\ref{eq: LSI in gen form}) with a constant $D^\beta_\Lambda(\ca)>0$ such that
\begin{equation}\label{eq: BD's raw result}
    D^\beta_\Lambda(\ca)\leq \frac{1}{4}+\frac{1}{2}\int_0^\beta \exp{\Big(2\int_0^t\chi_{s}^{(\Lambda)}(\ca) ds \Big)}dt,
\end{equation}
where for $\beta\geq 0$, $\chi_\beta(\ca)$ denotes the susceptibility of the measure $\tau^{\beta}_\Lambda$, i.e., 
\begin{equation}
    \chi_\beta(\ca):=\sup_{j\in\Lambda}\sum_{i\in\Lambda}\tau^\beta_\Lambda(\sigma_i\sigma_j).
\end{equation}

\end{thm}

By the second Griffiths inequality, 
for $0\leq\beta'\leq \beta''$ and every $A\subset\Lambda$, one has $\tau_\Lambda^{\beta''}(\sigma_A)\geq \tau_\Lambda^{\beta'}(\sigma_A)$, thus it follows that the susceptibility $\chi_\beta(\ca)$ is a non-decreasing function of $\beta$, i.e., for all $\beta'\leq \beta''$, $\chi_{\beta'}(\ca)\leq \chi_{\beta''}(\ca)$.
Thus one obtains for the log-Sobolev constant $D^\beta_\Lambda(\ca)$ of $\tau_\Lambda^\beta$ from (\ref{eq: BD's raw result}) that
\begin{equation}\label{BD-LSI with better constant}
    D^\beta_\Lambda(\ca)\leq \frac{1}{4}+\frac{\beta}{2} e^{2\beta \chi_\beta(\ca)}.
\end{equation}
\begin{proof}[Proof of Theorem \ref{thm: LSI for infnty volume Gibbs meas with uniform constant}]
Fix a finite volume $\Lambda\Subset\Z$.
Note that the coupling constants of a Dyson interaction $\Phi$ (\ref{eq: Dyson interaction}) form a matrix $\ca^{(\infty)}(\Lambda)$ which is given by 
\begin{equation}
    \ca^{(\infty)}(\Lambda)_{ij}:=\begin{cases}
         -J(|i-j|),\; & i\neq j, \; i,j\in\Lambda,\\
         0, & i=j.
    \end{cases}
\end{equation} 
Similarly, the coupling matrix $\ca^{(k)}(\Lambda)$ of an intermediate interaction $\Psi^{(k)}$ such that $\Lambda_i\subset\Lambda$, $i=\overline{1,k}$ is given by
\begin{equation}
    \ca^{(k)}(\Lambda)_{ij}:=\begin{cases}
        -J(|i-j|),\; & i\neq j,\; i,j\in\Lambda,\; \{i,j\}\in \A^{(k)};\\
        0,\; & \text{otherwise},
    \end{cases}
\end{equation}
where 
$
\A^{(k)}
:=
\{\{i,j\}\subset\Z: i,j\geq 0 \text{ or } i,j<0 \}
\cup 
\{\Lambda_1,\Lambda_2, \dots \Lambda_k\}.
$
Note that
\begin{equation}
    \sup_{j\in\Lambda}\sum_{i\in\Lambda}|\ca^{(k)}(\Lambda)_{ij}|
    \leq
    \sup_{j\in\Lambda}\sum_{i\in\Lambda}|\ca^{(\infty)}(\Lambda)_{ij}|
    =
    \sup_{j\in\Lambda}\sum_{i\in\Lambda\setminus\{j\}}J(|i-j|)
    <
    2\sum_{i=1}^\infty J(i)<\infty.
\end{equation}
Thus, by Schur's theorem, $l^2$ norms of the matrices $\ca^{(\infty)}(\Lambda)$ and $\ca^{(k)}(\Lambda)$ are bounded by $\kappa:=2\sum_{i=1}^\infty J(i)$.
Consequently, since $\ca^{(\infty)}(\Lambda)$ and $\ca^{(k)}(\Lambda)$ are symmetric matrices, $\ca^{(\infty)}(\Lambda)+\kappa I_{\Lambda\times\Lambda}$ and $\ca^{(k)}(\Lambda)+\kappa I_{\Lambda\times\Lambda}$ are positive definite, where $ I_{\Lambda\times\Lambda}$ is the $\Lambda\times\Lambda$ identity matrix.
Then, one can verify that for every $k\geq 0$, the matrices 
\begin{equation}
    \tilde \ca^{(\infty)}(\Lambda):=\frac{1}{2\kappa}\Big(\ca^{(\infty)}(\Lambda)+\kappa I_{\Lambda\times\Lambda}\Big)
\end{equation}

\begin{equation}
    \tilde \ca^{(k)}(\Lambda):=\frac{1}{2\kappa}\Big(\ca^{(k)}(\Lambda)+\kappa I_{\Lambda\times\Lambda}\Big)
\end{equation}
satisfy the conditions (C1)-(C3) in Theorem \ref{thm: BD result}.
Since $\frac{1}{2} I_{\Lambda\times\Lambda}$ participates both in the numerator and the denominator in the Boltzmann ansatz (\ref{Boltzmann ansatz for matrix A}), it cancels out. 
Therefore, for every $\beta\geq 0$, the matrices $\beta\tilde\ca^{(\infty)}(\Lambda)$ and $\beta\tilde \ca^{(k)}(\Lambda)$ produce, via (\ref{Boltzmann ansatz for matrix A}), the finite-volume Gibbs measures $\mu_\Lambda^{\beta/2\kappa}$ and $\nu_\Lambda^{(k), \beta/2\kappa}$ of the interactions $\frac{\beta}{2\kappa}\Phi$ and $\frac{\beta}{2\kappa}\Psi^{(k)}$ with the free boundary conditions, respectively.
By the second Griffiths inequality, for every $\beta<2\kappa\beta_c(\Phi)$ and $A\subset\Lambda$ one has that 
\begin{equation}
    \nu_\Lambda^{(k), \beta/2\kappa}(\sigma_A)
    \leq
    \mu_\Lambda^{\beta/2\kappa}(\sigma_A)
    \leq
    \mu^{\beta/2\kappa}(\sigma_A),
\end{equation}
where $\mu^{\beta/2\kappa}$ is the unique infinite-volume Gibbs measure for $\frac{\beta}{2\kappa}\Phi$.
Thus
\begin{eqnarray}\label{universal bound for finite vol suscep}
    \chi_{\beta/2\kappa}(\ca^{(k)}(\Lambda))
    \leq
    \chi_{\beta/2\kappa}(\ca^{(\infty)}(\Lambda))
    \leq
    \chi_{\beta/2\kappa}(\Phi)
    :=
    \sum_{i\in\Z}\mu^{\beta/2\kappa}(\sigma_0\sigma_i).
\end{eqnarray}
Consider a local function $f:X\to\R$ such that $\dom(f)\subset\Lambda$, where $\dom(f):=\{i\in\Z: \delta_i f\neq 0\}$.
Then Theorem \ref{thm: BD result} together with (\ref{BD-LSI with better constant}) and (\ref{universal bound for finite vol suscep}) yields,
\begin{equation}
    \texttt{Ent}_{\mu_\Lambda^{\beta/2\kappa}}(f^2)
    \leq 2D\int_{X_\Lambda}\sum_{i\in\Lambda} (f(\omega)-f(\omega^{(i)}))^2 d\mu^{\beta/2\kappa}_\Lambda,
\end{equation}
where $D:=\frac{1}{4}+\frac{\beta}{4\kappa}\exp{(\frac{\beta}{\kappa} \chi_{\beta/2\kappa}(\Phi))}$,
which is finite, by Theorem 2.1 in \cite{D-CT2016}, if $\beta<2\kappa\beta_c(\Phi)$. 
Thus since $\dom(f)\subset\Lambda$,
\begin{multline}\label{LSI for mu_Lambda with uniform constant}
    \int_{X} f^2 \log f^2 \, d(\mu^{\beta/2\kappa}_\Lambda\times \rho_{0}^{\Lambda^c}) 
    - 
    \int_{X} f^2 \, d(\mu^{\beta/2\kappa}_\Lambda\times \rho_{0}^{\Lambda^c}) \cdot \log \int_{X} f^2 \, d(\mu^{\beta/2\kappa}_\Lambda\times \rho_{0}^{\Lambda^c})\\
    \leq
    2D\int_{X}\sum_{i\in\Z} (f(\omega)-f(\omega^{(i)}))^2 d(\mu^{\beta/2\kappa}_\Lambda\times \rho_{0}^{\Lambda^c}),
\end{multline}
where $\rho$ is the uniform Bernoulli measure on $E=\{\pm1\}$, i.e., $\rho_0:=\frac{1}{2}\delta_{-1}+\frac{1}{2}\delta_{1}$.
Theorem \ref{thm: BD result} along with (\ref{BD-LSI with better constant})  and (\ref{universal bound for finite vol suscep}) also implies the log-Sobolev inequalities for the measures $\nu_\Lambda^{(k), \beta/2\kappa}$ with the universal constant $D:=\frac{1}{4}+\frac{\beta}{4\kappa}\exp{(\frac{\beta}{\kappa} \chi_{\beta/2\kappa}(\Phi))}$:
\begin{multline}\label{LSI for nu^(k)_Lambda with uniform constant}
    \int_{X} f^2 \log f^2 \, d(\nu_\Lambda^{(k), \beta/2\kappa}\times \rho_{0}^{\Lambda^c}) 
    - 
    \int_{X} f^2 \, d(\nu_\Lambda^{(k), \beta/2\kappa}\times \rho_{0}^{\Lambda^c}) \cdot \log \int_{X} f^2 \, d(\nu_\Lambda^{(k), \beta/2\kappa}\times \rho_{0}^{\Lambda^c})\\
    \leq
    2D\int_{X}\sum_{i\in\Z} (f(\omega)-f(\omega^{(i)}))^2 d(\nu_\Lambda^{(k), \beta/2\kappa}\times \rho_{0}^{\Lambda^c}).
\end{multline}
Note that for every $\beta<2\kappa\beta_c(\Phi)$, since there exists a unique
Gibbs measure $\mu^{\beta/2\kappa}$ for the interaction $\Phi$, the finite-volume measures $\mu^{\beta/2\kappa}_\Lambda\times \rho_{0}^{\Lambda^c}$ converge to $\mu^{\beta/2\kappa}$ as $\Lambda\uparrow\Z$ in the weak$^*$ topology. 
Analogously, by Proposition \ref{prop: critical beta's for dyson and intermediate interactions}, for each $\beta<2\kappa\beta_c(\Phi)$, the finite-volume measures $\nu_\Lambda^{(k), \beta/2\kappa}\times \rho_{0}^{\Lambda^c}$ converge to the unique infinite-volume Gibbs measure $\nu^{(k), \beta/2\kappa}\in\G(\Psi^{(k)})$ as $\Lambda\uparrow\Z$.
Hence, by taking limit as $\Lambda\uparrow\Z$ in (\ref{LSI for mu_Lambda with uniform constant}) and (\ref{LSI for nu^(k)_Lambda with uniform constant}), one obtains that for every $\beta<2\kappa\beta_c(\Phi)$ the measures $\mu^{\beta/2\kappa}$ and $\nu^{(k), \beta/2\kappa}$, $k\geq 0$, satisfy the log-Sobolev inequality with a uniform constant $D=\frac{1}{4}+\frac{\beta}{4\kappa}\exp{(\frac{\beta}{\kappa} \chi_{\beta/2\kappa}(\Phi))}$.
Thus, after rescaling $\frac{\beta}{2\kappa}\to \beta$, one concludes the statement of Theorem \ref{thm: LSI for infnty volume Gibbs meas with uniform constant}, i.e., for $\tilde D:=\frac{1}{4}+\frac{\beta}{2}\exp{(2\beta \chi_{\beta}(\Phi))}$ and for every $\beta<\beta_c(\Phi)$, $k\in\Z_+$ and for all local functions $f: X\to\R$,
\begin{equation}
    \int\limits_{X} f^2 \log f^2 \, d\mu^{\beta} 
    - 
    \int\limits_{X} f^2 \, d\mu^{\beta} \cdot \log \int\limits_{X} f^2 \, d\mu^{\beta}
    \leq
    2\tilde D\int\limits_{X}\sum_{i\in\Z} (f(\omega)-f(\omega^{(i)}))^2 d\mu^{\beta},
\end{equation}
and 
\begin{equation}
    \int\limits_{X} f^2 \log f^2 \, d\nu^{(k), \beta} 
    - 
    \int\limits_{X} f^2 \, d\nu^{(k), \beta} \cdot \log \int\limits_{X} f^2 \, d\nu^{(k), \beta}
    \leq
    2\tilde D\int\limits_{X}\sum_{i\in\Z} (f(\omega)-f(\omega^{(i)}))^2 d\nu^{(k), \beta}.
\end{equation}

\end{proof}

\section{Proof of Theorem \ref{main result}}\label{Proof of the main result}

\subsection{The first part}
By Theorem \ref{thm: LSI for infnty volume Gibbs meas with uniform constant} and Proposition \ref{prop: GCB from LSI}, the Gaussian Concentration Bounds hold for the unique Gibbs measure $\mu$ of $\beta\Phi$, as well as for the Gibbs measures $\nu^{(k)}$ corresponding to the intermediate interactions $\beta\Psi^{(k)}$, for all $k \in \Z_+$ and every $\beta \in [0, \beta_c(\Phi))$, with a uniform GCB constant given by
\begin{equation}\label{The Universal GCB constant}
   D := \frac{1}{8}\Big(1 + 2\beta\exp(2\beta \chi_{\beta}(\Phi))\Big)
    \cdot
    \Big(e^{4\beta\sum_{k=1}^\infty J(k)} + 1\Big).
\end{equation}
Since $\nu^{(0)} = \nu_- \times \nu$, where $\nu_-$ denotes the unique Gibbs measure for the restricted interaction $\beta\Phi|_{-\N}$, the measure $\nu$ inherits the GCB from $\nu^{(0)}$.

\subsection{The second part}
The proof that we provide below is similar to the proof of Theorem E in \cite{EFMV2024}.
The main difference between the proofs is that in \cite{EFMV2024}, the spin-spin correlations are estimated using the Dobrushin interdependence matrix. 
Since the Dobrushin interdependence matrix is not available in the near-critical regime, we instead treat the total sum of spin-spin correlations, which equally serves well for our purpose, using the result by Duminil-Copin and Tassion \cite{D-CT2016}.

In the light of Proposition 3.1 in \cite{EFMV2024}, it suffices to show that $\mu|_{\Z_+}\ll \nu$, which, in turn, follows from $\mu\ll \nu^{(0)}$.
Fix $\beta\in [0,\beta_c(\Phi))$.
Then since $\beta_c(\Phi)\leq \beta_c(\Psi^{(k)})$, for every $k\geq 0$, the interaction $\beta\Psi^{(k)}$ has a unique Gibbs measure $\nu^{(k)}$.
Since each $\beta\Psi^{(k)}$ is a finite perturbation of $\beta\Psi^{(0)}$ (\cite{Georgii-book}*{Chapter 7}), $\nu^{(k)}$ is absolutely continous with respect to $\nu^{(0)}$ with a density 
\begin{equation}\label{density of nu^(k) wrt nu^(0)}
    \frac{d\nu^{(k)}}{d\nu^{(0)}}
    =
    \frac{e^{-W_k}}{\int_X e^{-W_k} d\nu^{(0)}},
\end{equation}
where $W_k:=\sum_{\iota=1}^k\beta\Phi_{\Lambda_\iota}$.
We denote the right hand side of (\ref{density of nu^(k) wrt nu^(0)}) by $f^{(k)}$.
It follows from Theorem D in \cite{EFMV2024} that if the family $\{f^{(k)}: k\in\N\}\subset L^1(X, \nu^{(0)})$ is uniformly integrable, then $\mu$ is absoluetly continuous with respect to $\nu^{(0)}$. 
To verify uniform integrability, we apply De la Vallée Poussin's theorem \cite{V-P1915}, which asserts that if 
\begin{equation}\label{place to apply Valle-Poissen}
    \sup_{k\geq 0} \int_{X}f^{(k)}\log f^{(k)}d\nu^{(0)}<\infty,\;
\end{equation}
then the family $\{f^{(k)}: k\in\N\}$ is uniformly integrable in $L^1(X,\nu^{(0)})$.
Note that
\begin{equation}\label{the sum}
    \int_{X}f^{(k)}\log f^{(k)}d\nu^{(0)}
    =
    -
    \int_{X}W_kd\nu^{(k)}
    -
    \log \int_{X}e^{-W_k}d\nu^{(0)},
\end{equation} 
hence $\{f^{(k)}: k\in\N\}$ is uniformly integrable if each term in the right-hand side of (\ref{the sum}) is bounded uniformly in $k$. 
One can check using the Griffiths inequalities that
\begin{eqnarray}
    \Big\lvert \int_X W_k d\nu^{(k)} \Big\rvert
    &=&\notag
    \sum_{\iota=1}^k\int_X -\beta\Phi_{\Lambda_\iota} d\nu^{(k)}\\
    &\leq&\notag
    \sum_{\iota=1}^k \int_X -\beta \Phi_{\Lambda_\iota}d\mu\\
    &\leq&\notag
    \sum_{i=1}^\infty\sum_{j=0}^{\infty}
    \int_{X} \beta J(i+j)\sigma_{-i}\sigma_j d\mu\\
    &=&\label{bounding integrals of W_k using GKS}
    \sum_{p=1}^\infty \beta J(p) \mu(\sigma_0\sigma_p)\cdot \# U_p,
\end{eqnarray}
where $U_p:=\{(i,j)\in \N\times \Z_+: i+j=p\}$ and in the last equation, we used the translation-invariance of the measure $\mu$.
Since $\# U_p=p$ and by assumption $J(p)=O(p^{-1})$, it follows from (\ref{bounding integrals of W_k using GKS}) that 
\begin{equation}\label{bounding integrals of W_k with susceptibility}
    \Big\lvert \int_X W_k d\nu^{(k)} \Big\rvert
    \leq 
    \sum_{p=1}^\infty \beta p J(p) \mu(\sigma_0 \sigma_p) 
    \leq 
    \beta\cdot \sup_{p\in\N} pJ(p)\cdot\sum_{p=1}^\infty  \mu(\sigma_0 \sigma_p)
    =
    \beta\cdot \sup_{p\in\N} pJ(p)\cdot\chi_\beta(\Phi).
\end{equation}
As $\chi_\beta(\Phi)<\infty$ for every $0\leq \beta<\beta_c(\Phi)$ \cite{D-CT2016}*{Theorem 2.1}, one obtains from (\ref{bounding integrals of W_k with susceptibility}) that $\sup_{k\in\N} \Big\lvert \int_X W_kd\nu^{(k)} \Big\rvert <\infty$. 

\noindent
We now show that the second term on the right-hand side of (\ref{the sum}) is bounded uniformly in $k\in\N$, more precisely, we prove that $\sup_{k\in\N}\int_{X}e^{-W_k}d\nu^{(0)}<\infty$.
By telescoping (\ref{density of nu^(k) wrt nu^(0)}), one can argue that (see also (49) in \cite{EFMV2024} and (36) in \cite{M2025})
\begin{equation}
    \int_X e^{-\beta(\Phi_{\Lambda_k}+\Phi_{\Lambda_{k-1}}+\dots+\Phi_{\Lambda_1})}d\nu^{(0)}
    =
    \prod_{\iota=1}^k\int_X e^{-\beta\Phi_{\Lambda_\iota}} d\nu^{(\iota-1)}.
\end{equation}
Thus by applying the GCB to the measure $\nu^{(\iota-1)}$ and the local function $e^{-\beta\Phi_{\Lambda_\iota}}$,
\begin{equation}\label{eq: bounding exp term in valle poussin}
    \int_X e^{-W_k}d\nu^{(0)}
    \leq
    e^{D\sum_{\iota=1}^k \lVert\underline{\delta}(\beta\Phi_{\Lambda_\iota}) \rVert_2^2}
    \cdot
    e^{\beta\sum_{\iota=1}^k\int_X-\Phi_{\Lambda_\iota}d\nu^{(\iota-1)}},
\end{equation}
where $D$ is given by (\ref{The Universal GCB constant}).
Note that 
\begin{equation}\label{int eigen bound on variations}
    \sum_{\iota=1}^k \lVert\underline{\delta}(\beta\Phi_{\Lambda_\iota}) \rVert_2^2
    \leq
    \sum_{\iota=1}^\infty \lVert\underline{\delta}(\beta\Phi_{\Lambda_\iota}) \rVert_2^2
    =
    \beta^2\sum_{i=1}^\infty\sum_{j=0}^\infty \lVert\underline{\delta}(\Phi_{\{-i,j\}}) \rVert_2^2.
\end{equation}
Since
\begin{equation*}
    \delta_s(\Phi_{\{-i,j\}})=\begin{cases} 
      0, & s\in \Z\setminus\{-i,j\}; \\
      2J(i+j), & s\in\{-i,j\},
     \end{cases}
\end{equation*}
one has
\begin{equation}
    \lVert\underline{\delta}(\Phi_{\{-i,j\}}) \rVert_2^2
    =
    \sum_{s\in\Z} (\delta_s(\Phi_{\{-i,j\}}))^2
    =
    8 J(i+j)^2.
\end{equation}
Thus by (\ref{int eigen bound on variations}), for all $k\in\N$,
\begin{equation}\label{eq: ultimate bound for oscillations}
    \sum_{\iota=1}^k \lVert\underline{\delta}(\beta\Phi_{\Lambda_\iota}) \rVert_2^2
    \leq
    \sum_{\iota=1}^\infty \lVert\underline{\delta}(\beta\Phi_{\Lambda_\iota}) \rVert_2^2
    =
    \sum_{i=1}^\infty\sum_{j=0}^\infty 8\beta^2 J(i+j)^2<\infty.
\end{equation}
Since $\Phi$ is a pair interaction, the first Griffiths inequality implies that for every $\iota\in\N$, 
$\int_X-\Phi_{\Lambda_\iota}d\nu^{(\iota-1)}\geq 0$ and $\int_X-\Phi_{\Lambda_\iota}d\mu\geq 0$.
Hence, by the Griffiths inequalities, one has
\begin{eqnarray}
    0\leq \beta\sum_{\iota=1}^k\int_X -\Phi_{\Lambda_\iota}d\nu^{(\iota-1)}
    &\leq&\notag
    \beta\sum_{\iota=1}^k\int_X -\Phi_{\Lambda_\iota}d\mu\\
    &\leq&\notag
    \beta\sum_{\iota=1}^\infty\int_X -\Phi_{\Lambda_\iota}d\mu\\
    &=&\notag
    \sum_{i=1}^\infty\sum_{j=0}^\infty \beta J(i+j)\mu(\sigma_{-i}\sigma_j)\\
    &\overset{(\ref{bounding integrals of W_k using GKS})}{=}&\notag
    \sum_{p=1}^\infty \beta J(p)\mu(\sigma_0\sigma_p) \cdot \# U_p\\
    &=&\notag
    \sum_{p=1}^\infty \beta p J(p) \mu(\sigma_0\sigma_p)\\
    &<&
    \beta \cdot \sup_{p\in\N} p J(p) \cdot \chi_\beta(\Phi).
\end{eqnarray}
Hence, by (\ref{eq: ultimate bound for oscillations}) and (\ref{eq: bounding exp term in valle poussin}), one concludes the desired claim. 

\subsection{The third part}
With Gaussian concentration bounds for $\nu^{(0)}$ now established across the entire interval  $\beta\in [0,\beta_c(\Phi))$, the proof of the continuity of the eigenfunctions replicates the argument previously used under the Dobrushin condition -- specifically, that of the third part of Theorem C in \cite{EFMV2024}. 
We note that the key contribution of this part of the paper is the establishment of the GCB for the unique Gibbs measure associated with the interaction $\Psi^{(0)}$, 
%with a uniform GCB constant
-- a result that was previously known only in the Dobrushin uniqueness region, which notably excludes the near-critical values of $\beta$.
For the reader’s convenience, we briefly summarise the main steps below. 
We adopt the notations from the proof of Theorem C in \cite{EFMV2024} and omit their definitions except where necessary for clarity.

The proof is based on the Arzela-Ascoli theorem. 
In the presentation below, $\xi$ stands for the configurations in the left half space $X_-:=\{\pm1\}^{-\N}$ and $\eta$ stands for the configurations in the right half space $X_+=\{\pm1\}^{\Z_+}$.
For simplicity of the presentation, we consider the subfamily $\{\Psi^{(k_N)}:N\in\N\}$ of the intermediate interactions, here $k_N$ is defined in (\ref{indexation rule of mathcal{A}}).
Then for the Radon-Nikodym density $f^{[N]}$ of the unique Gibbs measure $\nu^{[N]}\in\G(\Psi^{(k_N)})$ with respect to $\nu^{(0)}$, one has 
\begin{equation}
    f^{[N]}=\frac{e^{-W_{[N]}}}{\int_{X}e^{-W_{[N]}}d\nu^{(0)}},\;\;\;\; W_{[N]}(\xi,\eta):=\sum_{i=1}^N\sum_{j=0}^N-\beta J(i+j)\xi_{-i}\eta_j.
\end{equation}
In particular, the density of the restriction $\nu^{[N]}|_{X_+}$ of $\nu^{[N]}$  to $X_+$ with respect to $\nu$ is given by 
\begin{equation}
    f_+^{[N]}(\eta)
    =
    \int_{X_-} f^{[N]}(\xi,\eta) \nu_-(d\xi)
    =
    \frac{\int_{X_-} e^{-W_{[N]}(\xi, \eta)}\nu_-(d\xi)}{\int_{X_+}\int_{X_-} e^{-W_{[N]}(\xi, \zeta)}\nu_-(d\xi)\nu(d\zeta)}.
\end{equation}
We argue that it suffices to show the relative compactness of $\{f_+^{[N]}: N\in\N\}$ in $C(X_+)$. 
Indeed, if that is the case, then there exists a subsequence $\{f_+^{[N_s]}\}_{s\in\N}$ converging to some $f_+\in C(X_+)$ in the uniform topology. 
Then, by the argument in the proof of Theorem D in \cite{EFMV2024}, $f_+$ is the Radon-Nikodym density of the restriction $\mu|_{X_+}$ of $\mu$ to $X_+$ with respect to $\nu$, i.e., $f_+=\frac{d\mu|_{X_+}}{d\nu}$.
In particular, by Proposition 3.1 in \cite{EFMV2024}, $f_+$ is the eigenfunction of the transfer operator $\La_\phi$, i.e., $\La_\phi f_+(\eta)=\lambda f_+(\eta)$ for $\nu-$almost every $\eta\in X_+$.
Thus, since $\nu$ is fully supported, one has $\La_\phi f_+(\eta)=\lambda f_+(\eta)$ for all $\eta\in X_+$.

Below, we aim to prove the relative compactness of the family $\{f_+^{[N]}: N\in\N\}\subset C(X_+)$.
Then, by the Arzela-Ascoli theorem, it suffices to show the uniform boundedness and the equicontinuity of the family.

\noindent
\underline{\textit{Uniform boundedness:}}
As $\nu^{(0)}=\nu_-\times \nu$, $\nu_-$ inherits the GCB property from $\nu^{(0)}$.
Note that the interaction $\Phi|_{X_-}=(\Phi_{\Lambda})_{\Lambda\Subset -\N}$ is invariant under a global spin-flip transformation, and so is the unique Gibbs measure $\nu_-\in \G(\Phi|_{X_-})$. 
Hence for all $N\in\N$ and $\eta\in X_+$, $\int_{X_-}W_{[N]}(\xi,\eta)\nu_-(d\xi)=0$. Thus, by the GCB, one has for every $\kappa\in\R$ that 
\begin{equation}\label{GCB to nu_-}
    \int_{X_-} e^{\kappa W_{[N]}(\xi,\eta)}\nu_-(d\xi)
    \leq e^{D \kappa^2 ||\underline \delta (W_{[N]}(\cdot,\eta))||_2^2}.
\end{equation}
Note that for all $i\in\N$, 
$$
\delta_{-i}(W_{[N]}(\cdot,\eta))
=
2\beta \Big| \sum_{j=0}^N J(i+j)\eta_j\Big|
\leq 
2\beta \sum_{j=0}^N J(i+j)
\leq
2\beta \sum_{j=i}^\infty J(j),
$$
hence
\begin{equation}\label{eq:r-variance}
    ||\underline \delta (W_{[N]}(\cdot,\eta))||_2^2\leq 4\beta^2 \sum_{i=1}^{\infty}\Big(\sum_{j=i}^{\infty} J(j) \Big)^2=:4\beta^2 C_1.
\end{equation}
By combining (\ref{eq:r-variance}) and (\ref{GCB to nu_-}), one obtains for every $N\in\N$, $\eta\in X_+$ and $\kappa\in\R$ that
\begin{equation}\label{UB for num of f^N_+}
    \int_{X_-} e^{\kappa W_{[N]}(\xi,\eta)}\nu_-(d\xi)\leq e^{4 D \kappa^2 \beta^2 C_1}. 
\end{equation}
Changing $\kappa\to-\kappa$ and applying the Cauchy-Schwarz inequality, one can also obtain the lower bound for the integral, i.e., 
\begin{equation}\label{LB for num of f^N_+}
    e^{-4 D \kappa^2 \beta^2 C_1}\leq \int_{X_-} e^{\kappa W_{[N]}(\xi,\eta)}\nu_-(d\xi).
\end{equation}
By combining (\ref{LB for num of f^N_+}) with (\ref{UB for num of f^N_+}), one concludes that
\begin{equation}\label{bounds for denom of f^N and f^N_+}
    e^{-4 D \beta^2 C_1}
    \leq \int_{X_+}\int_{X_-} e^{-W_{[N]}(\xi,\eta)}\nu_-(d\xi)\nu(d\eta)
    \leq e^{4 D \beta^2 C_1} 
\end{equation}
and also for every $\eta\in X_+$ that
\begin{equation}\label{uniform ub lb boundedness of half-line densities}
    e^{-8 D \beta^2 C_1}
    \leq f^{[N]}_+(\eta)
    \leq e^{8 D \beta^2 C_1}.
\end{equation}

\noindent
\underline{\textit{Equicontinuity:}}
As the denominator $\int_X e^{-W_{[N]}}d\nu^{(0)}$ is uniformly bounded as shown in (\ref{bounds for denom of f^N and f^N_+}), it suffices to show that the family $\Big\{\int_{X_-}e^{-W_{[N]}(\xi,\cdot)}\nu_-(d\xi): N\in\N\Big\}$  is equicontinuous. 
Consider $n\in\N$ and configurations $\eta, \tilde \eta\in X_+$ such that $\eta_0^{n-1}=\tilde\eta _0^{n-1}$.
Then, by the Cauchy-Schwarz inequality,
\begin{multline}
    \Big|\int_{X_-}\Big[e^{-W_{[N]}(\xi,\eta)}-e^{-W_{[N]}(\xi,\tilde\eta)} \Big]\nu_-(d\xi)\Big|\\
    \leq \Big(\int_{X_-}e^{-2W_{[N]}(\xi,\eta)}\nu_-(d\xi) \Big)^{\frac{1}{2}} \Big(\int_{X_-}\Big[ e^{W_{[N]}(\xi,\eta)-W_{[N]}(\xi,\tilde\eta)}-1 \Big]^2 \nu_-(d\xi)\Big)^{\frac{1}{2}}
\end{multline}
and (\ref{UB for num of f^N_+}) yields that
\begin{equation}\label{equicontinuity important bound}
    \Big|\int_{X_-}\Big[e^{-W_{[N]}(\xi,\eta)}-e^{-W_{[N]}(\xi,\tilde\eta)} \Big]\nu_-(d\xi)\Big|
    \leq 
    C_2 \Big(\int_{X_-}\Big[ e^{W_{[N]}(\xi,\eta)-W_{[N]}(\xi,\tilde\eta)}-1 \Big]^2\nu_-(d\xi) \Big)^{\frac{1}{2}}
\end{equation}
with $C_2:= e^{8 D  \beta^2 C_1}$.  
We will bound the last integral by bounding the exponent.
\smallskip

Note that
for all $i\in\N$,
\begin{equation}
    \delta_{-i}(W_{[N]}(\cdot,\eta)-W_{[N]}(\cdot,\tilde\eta))
    =
    2\beta\Big|\sum_{j=n}^N J(i+j) (\eta_j-\tilde\eta_j) \Big|
    \leq
    4\beta \sum_{j=n}^N J(i+j)\;.
\end{equation}
Hence, for sufficiently large $n$ and $N>n$,
\begin{equation*}\label{oscillation estimate for the difference of W_N's}
    ||\underline\delta(W_{[N]}(\cdot,\eta)-W_{[N]}(\cdot,\tilde\eta))||_2^2
    \leq 16\beta^2\sum_{i=1}^{\infty}\Big(\sum_{j=n}^N J(i+j)\Big)^2
    \leq 
    16 \beta^2 \sum_{i=1}^{\infty}
    \Big(\sum_{j=n+i}^\infty J(j) \Big)^2
    =:u_n
\end{equation*}
with 
$
\lim\limits_{n\to\infty}u_n=0.
$
The spin-flip invariance of $\nu_-$ implies that for all $N\in\N$,
\[
\int_{X_-} [W_{[N]}(\xi,\eta)-W_{[N]}(\xi,\tilde\eta)]\,\nu_-(d\xi)=0\;.
\]
Then, since $\nu_-$ satisfies the GCB, we have, from the moment concentration bounds (\ref{eq: MCB}), that for all $m\in\N$,
\begin{eqnarray}\label{apply MCB to difference of W_N}
    \int_{X_-}\Big|W_{[N]}(\xi,\eta)-W_{[N]}(\xi,\tilde\eta) \Big|^m\nu_-(d\xi)
   & \leq& \Big(\frac{D  ||\underline\delta(W_{[N]}(\cdot,\eta)-W_{[N]}(\cdot,\tilde\eta))||_2^2   }{2} \Big)^{\frac{m}{2}}m\,\Gamma\Big (\frac{m}{2}\Big) \nonumber\\
   &=&  m\, v_n^m\, \Gamma \Big(\frac{m}{2}\Big)
\end{eqnarray}
with
\[
v_n:=\Big(\frac{D\,u_n}{2} \Big)^{\frac{1}{2}}\;.
\]
To conclude, we expand the square on the right-hand side of (\ref{equicontinuity important bound}):
\begin{eqnarray}
 \lefteqn{\int_{X_-}\Big[ e^{W_{[N]}(\xi,\eta)-W_{[N]}(\xi,\tilde\eta)}-1 \Big]^2 \nu_-(d\xi)}\nonumber\\
 &=& \int_{X_-}\Big[ e^{2\left[W_{[N]}(\xi,\eta)-W_{[N]}(\xi,\tilde\eta)\right]}- 2\, e^{W_{[N]}(\xi,\eta)-W_{[N]}(\xi,\tilde\eta)} 
  +1 \Big]\nu_-(d\xi)\nonumber\\
 &\le& 
 \sum_{m=1}^{\infty} \frac{2^m+2}{m!} \int_{X_-}\Bigl|W_{[N]}(\xi,\eta)-W_{[N]}(\xi,\tilde\eta) \Bigr|^m\nu_-(d\xi).
\end{eqnarray}
Using elementary analysis, one can demonstrate that
\begin{equation}
    \sum_{m=1}^{\infty}\frac{2^m+2}{m!} m \,v_n^m\, \Gamma \Big(\frac{m}{2}\Big)
    \leq 
    (6v_n^2+8v_n)e^{v_n^2}.
\end{equation}
Therefore, we obtain from (\ref{equicontinuity important bound}) and \eqref{apply MCB to difference of W_N}, 
\begin{equation}
    \Big|\int_{X_-}\Big[e^{-W_{[N]}(\xi,\eta)}-e^{-W_{[N]}(\xi,\tilde\eta)} \Big]\nu_-(d\xi)\Big|
    \leq
    C_2 
    ((6v_n^2+8v_n)e^{v_n^2})^{\frac{1}{2}}.
\end{equation}
Since $\lim\limits_{n\to\infty}v_n=0$, we conclude the claim.

\section*{Appendix: GCB from LSI in the Gibbsian setting}
We call an interaction $\Psi=(\Psi_\Lambda)_{\Lambda\Subset\Z}$ on $X=\{\pm 1\}^\Z$ \textit{strongly uniformly absolutely convergent} (SUAC) if 
\begin{equation}
    \lVert\Psi\rVert := \sup_{i\in\Z}\sum_{\substack{V\ni i\\ V\Subset\Z}} \lVert \Psi_V\rVert_\infty<\infty.
\end{equation}

The goal of the appendix is to show the following statement.

\begin{manualprop}{A.1}\label{prop: GCB from LSI}
    Assume $\mu$ is a unique Gibbs measure for a SUAC interaction $\Psi$.
    If $\mu$ satisfies the log Sobolev inequality with a constant $D>0$, then $\mu$ satisfies the Gaussian concentration bound with the constant $\frac{D(e^{2\lVert \Psi\rVert}+1)}{2}$.
\end{manualprop}

Note that for $i\in\Z$, the spin-flip transformation $\s_i:X\to X$ at $i$ is defined as $\s_i(\omega)=\omega^{(i)}$, $\omega\in X$.
We use the following lemma in the proof of Proposition \ref{prop: GCB from LSI}.

\begin{manuallem}{A.2}\label{imp lem for GCB from LSI}
    For every $i\in\Z$, the pushforward $\mu^{(i)}:=(\s_i)_*(\mu)=\mu\circ (\s_i)^{-1}$ of $\mu$ by $\s_i$ is absolutely continuous with respect to $\mu$ and $\frac{d\mu^{(i)}}{d\mu}\in L^{\infty}(\mu)$. 
    Furthermore, 
    \begin{equation}\label{uni bound for densities of flipping}
        \sup_{i\in\Z} \Big \lVert \frac{d\mu^{(i)}}{d\mu}\Big\rVert_{L^{\infty}(\mu)}\leq e^{2\lVert\Psi\rVert}. 
    \end{equation}
\end{manuallem}
\begin{proof}[Proof of Lemma \ref{imp lem for GCB from LSI}]
    Fix $i\in\Z$.  
    We consider finite volumes $\Lambda\ni i$ and corresponding the finite-volume Gibbs measures for $\Psi$ with the $+$ boundary condition:
    \begin{equation}\label{Appendix: Boltzmann ansatz}
        \mu_\Lambda([\omega_\Lambda])=\frac{e^{-H_\Lambda(\omega_\Lambda+_{\Lambda^c})}}{Z_\Lambda(+)},\;\; \omega_\Lambda\in \{\pm 1\}^\Lambda,
    \end{equation}
where $H_\Lambda:=\sum_{V\cap\Lambda\neq \emptyset}\Psi_V$ is the Hamiltonian in $\Lambda$ and $Z_\Lambda(+):=\sum_{\omega_\Lambda\in E^\Lambda}e^{-H_\Lambda(\omega_\Lambda+_{\Lambda^c})}$ is the corresponding partition function.
By the uniqueness, the finite-volume Gibbs measures $\mu_\Lambda$ converge to $\mu$ as $\Lambda\uparrow\Z$ in the weak star topology. 
Thus, by the continuity of the pushforward $(\s_i)_*$, one has $\mu_\Lambda^{(i)}\to\mu^{(i)}$, where $\mu_\Lambda^{(i)}:=\mu_\Lambda\circ (\s_i)^{-1}$.
Note that $\mu_\Lambda^{(i)}\ll \mu_\Lambda$ and for every $\omega_\Lambda\in E^{\Lambda}$, $\mu_\Lambda^{(i)}([\omega_\Lambda])=\mu_\Lambda([\omega^{(i)}_\Lambda])$.
Thus by (\ref{Appendix: Boltzmann ansatz}),
\begin{equation}\label{appendix: RN densities of finite volume measures}
    h^{(i)}_\Lambda(\omega)
    :=
    \frac{d\mu_\Lambda^{(i)}}{d\mu_\Lambda}(\omega)
    =
    \frac{\mu_\Lambda^{(i)}([\omega_\Lambda])}{\mu_\Lambda([\omega_\Lambda])}
    =
    e^{H_\Lambda(\omega_\Lambda+_{\Lambda^c})-H_\Lambda(\omega^{(i)}_\Lambda+_{\Lambda^c})}.
\end{equation}
Note that 
\begin{eqnarray}
    H_\Lambda(\omega_\Lambda+_{\Lambda^c})-H_{\Lambda}(\omega^{(i)}_\Lambda+_{\Lambda^c})
    &=&\notag
    \sum_{\substack{V\ni i}} [\Psi_V(\omega_\Lambda+_{\Lambda^c})-\Psi_V(\omega^{(i)}_\Lambda+_{\Lambda^c})]\\
    &=&\label{eq for Hamiltonians in Lambda and {i}}
    H_{\{i\}}(\omega_\Lambda+_{\Lambda^c})
    -
    H_{\{i\}}(\omega^{(i)}_\Lambda+_{\Lambda^c})
\end{eqnarray}
and by the continuity of $H_{\{i\}}$, the right-hand side of (\ref{eq for Hamiltonians in Lambda and {i}}) converges uniformly in $\omega\in X$ to $H_{\{i\}}(\omega)-H_{\{i\}}(\omega^{(i)})$ as $\Lambda\uparrow\Z$.
Therefore, it follows from (\ref{appendix: RN densities of finite volume measures}) that the finite-volume Radon-Nikodym densities $h_\Lambda^{(i)}(\omega)$ converge to 
\begin{equation}\label{eq: density h^{(i)}}
    h^{(i)}(\omega):=e^{H_{\{i\}}(\omega)-H_{\{i\}}(\omega^{(i)})}
\end{equation}
uniformly in $\omega\in X$ as $\Lambda\uparrow\Z$.
Since $\mu_\Lambda^{(i)}\overset{*}{\rightharpoonup}\mu^{(i)}$ as $\Lambda\uparrow\Z$, for any test function $f\in C(X)$,
\begin{equation}\label{eq: conv of mu_Lambda^i to mu^i}
    \lim_{\Lambda\uparrow\Z} \int_X fd\mu_\Lambda^{(i)}=\int_X f d\mu^{(i)}.
\end{equation}
On the other hand, 
\begin{equation}
\int_X f d\mu_\Lambda^{(i)}
    =
    \int_X f h_\Lambda^{(i)}d\mu_\Lambda
    =
    \int_X f h^{(i)} d\mu_\Lambda
    +
    \int_X f(h_\Lambda^{(i)}-h^{(i)})d\mu_\Lambda.    
\end{equation}
Then since $\mu_\Lambda\overset{*}{\rightharpoonup}\mu$ as $\Lambda\uparrow\Z$ and $\lVert h_\Lambda^{(i)} -h^{(i)} \rVert_\infty\xrightarrow[\Lambda\uparrow\Z]{}0 $,
one obtains
\begin{equation}
    \lim_{\Lambda\uparrow\Z}\int_X fd\mu_\Lambda^{(i)}
    =
    \int_X fh^{(i)} d\mu.
\end{equation}
Then by combining this with (\ref{eq: conv of mu_Lambda^i to mu^i}), one concludes that $\mu^{(i)}\ll \mu$ and 
\begin{equation}
    \frac{d\mu^{(i)}}{d\mu}=h^{(i)}.
\end{equation}
One can readily check from (\ref{eq: density h^{(i)}}) that for every $i\in\Z$, $\lVert h^{(i)}\rVert_\infty\leq e^{2\sum_{V\ni i} \lVert \Psi_V\rVert_\infty}\leq e^{2\lVert\Psi \rVert}$, and hence one concludes the last claim of the lemma.
    
\end{proof}

In the proof of Proposition \ref{prop: GCB from LSI}, we also use the following elementary lemma, which easily follows from the mean value theorem.
\begin{manuallem}{A.3}\label{elment lemma-mvt}
    For all $a,b\in\R$, one has $\lvert e^b-e^a\rvert\leq e^{\max\{a,b\}}|b-a|$.
\end{manuallem}

\begin{proof}[Proof of Proposition \ref{prop: GCB from LSI}]
Fix a local function $f:X\to\R$ and a number $\lambda\in [0,1]$. 
Consider $e^{\lambda f}$, and denote $\z(\lambda):=\int_X e^{\lambda f}d\mu$.
Note that $\z\in C^1([0,1])$ and for all $\lambda\in [0,1]$, $\z'(\lambda)=\int_X f e^{\lambda f}d\mu$.
By applying the log Sobolev inequality to $e^{\lambda f/2}$, one gets
\begin{equation}\label{LSI applied to e^{lambda f/2}}
    \lambda \int_X f e^{\lambda f} d\mu
    -
    \int_X e^{\lambda f} d\mu \cdot \log \Big( \int_X e^{\lambda f} d\mu\Big)
    \leq 
    2D\int_X \sum_{i\in\Z} \Big\lvert e^{\lambda f(\omega)/2}-e^{\lambda f(\omega^{(i)})/2} \Big\rvert^2 \mu(d\omega).
\end{equation}
By rewriting (\ref{LSI applied to e^{lambda f/2}}) in terms of $\mathcal{Z}$ and by employing Lemma \ref{imp lem for GCB from LSI} and Lemma \ref{elment lemma-mvt}, 
\begin{eqnarray}
    \lambda\z'(\lambda)-\z(\lambda)\log\z(\lambda)
    &\leq&\notag
    2D\int_X \sum_{i\in\Z} \Big\lvert e^{\lambda f(\omega)/2}-e^{\lambda f(\omega^{(i)})/2} \Big\rvert^2 \mu(d\omega)\\
    &\overset{\text{Lemma } \ref{elment lemma-mvt}}{\leq}&\notag
    2D\int_X \sum_{i\in\Z} e^{\lambda \max\{f(\omega), f(\omega^{(i)})\}} \frac{\lambda^2}{4} (\delta_if)^2 \mu(d\omega)\\
    &\leq&\notag
    \frac{D\lambda^2}{2} \sum_{i\in\Z} \Big[\int_X \Big( e^{\lambda f} + e^{\lambda f\circ \s_i}\Big) d\mu \cdot (\delta_i f)^2\Big]\\
    &\leq&\notag
    \frac{D\lambda^2}{2} \sum_{i\in\Z}
    \Big[
    \z(\lambda)
    +
    \int_X e^{\lambda f}d\mu^{(i)} 
    \Big]\cdot(\delta_i f)^2\\
    &\overset{\text{Lemma } \ref{imp lem for GCB from LSI}}{=}&\notag
    \frac{D\lambda^2}{2} \sum_{i\in\Z}
    \Big[
    \z(\lambda)
    +
    \int_X e^{\lambda f} h^{(i)}d\mu 
    \Big]\cdot(\delta_i f)^2\\
    &\overset{(\ref{uni bound for densities of flipping})}{\leq}&\notag
    \frac{D\lambda^2}{2} \sum_{i\in\Z}
    \Big[
    \z(\lambda)
    +
    e^{2\lVert\Psi\rVert} \z(\lambda)
    \Big]\cdot(\delta_i f)^2\\
    &=& \label{imp ineq for herbst arg}
    \frac{D(e^{2\lVert \Psi\rVert}+1)}{2}\cdot \lVert \underline{\delta}(f)\rVert_2^2\cdot \z(\lambda) \cdot \lambda^2.
\end{eqnarray}
Now consider a function $u$ on $[0,1]$ defined by $u(\lambda)=\frac{1}{\lambda}\log \z(\lambda)$ for  $\lambda\in (0,1]$, and $u(0)=\z'(0)=\int_X f d\mu$.
Note that $u\in C^1([0,1])$.
One can easily check that (\ref{imp ineq for herbst arg}) can be written in terms of $u$ as
\begin{equation}
    u'(\lambda)
    \leq
    \frac{D(e^{2\lVert \Psi\rVert}+1)}{2}\lVert \underline{\delta}(f)\rVert_2^2, \;\; \lambda\in [0,1].
\end{equation}
Hence 
\begin{equation*}
    u(1)-u(0)
    \leq
    \frac{D(e^{2\lVert \Psi\rVert}+1)}{2}\lVert \underline{\delta}(f)\rVert_2^2,
\end{equation*}
which implies the Gaussian concentration bounds since 
\begin{equation}
    u(1)-u(0)=\log \z(1)-\mu(f)=\log \Big(\int_X e^{f-\mu(f)}d\mu \Big)\leq \frac{D(e^{2\lVert \Psi\rVert}+1)}{2}\lVert \underline{\delta}(f)\rVert_2^2.
\end{equation}

\end{proof}

\section*{Acknowledgements}

The author is grateful to Aernout van Enter, Roberto Fernández, and Evgeny Verbitskiy for earlier collaborations that provided important insights relevant to this work. He also thanks Jean-René Chazottes and Evgeny Verbitskiy for valuable correspondence. Special appreciation is extended to Aernout van Enter for a careful reading of an early draft and for insightful suggestions that improved the presentation of the paper.

\end{document}